\documentclass[a4paper,11pt]{article}       
\usepackage{enumerate}          
\usepackage{amssymb,amsmath,amsfonts,amsxtra,amsthm}            
\usepackage{pdfsync}
\usepackage{url}
\usepackage{hyperref}
\usepackage{anysize}
\usepackage{graphicx,float,color}
\usepackage{epsfig}
\usepackage{tikz}

\graphicspath{{fig/}}

\begin{document}
\theoremstyle{plain}
\newtheorem{definition}{Definition}
\newtheorem{lemma}{Lemma}
\newtheorem{theorem}{Theorem}
\newtheorem{metatheorem}[definition]{(Meta)-theorem}
\newtheorem{proposition}[lemma]{Proposition}
\newtheorem{corollary}[definition]{Corollary}
\newtheorem{conjecture}{Conjecture}
\newtheorem{fact}{Fact}
\newtheorem{problem}[definition]{Problem}
\newtheorem{assumption}[definition]{Assumption}
\newtheorem{question}[definition]{Question}
\theoremstyle{definition}
\newtheorem{remark}{Remark}
\newtheorem{notation}[definition]{Notation}
\errorcontextlines=0

\newcommand{\R}{\mathbb{R}}
\newcommand{\Z}{\mathbb{Z}}
\newcommand{\Inv}{\text{Inv}}
\newcommand{\e}{\varepsilon}
\newcommand{\vol}{\text{vol}}
\newcommand{\tor}{(2\mathbb{T})^{2}}
\newcommand{\er}{{\overrightarrow{e_{r}}}}
\newcommand{\et}{{\overrightarrow{e_{\theta}}}}
\newcommand{\conv}{\text{conv}}
\renewcommand{\leq}{\leqslant}
\renewcommand{\geq}{\geqslant}
\newcommand{\todo}[1]{$\clubsuit$ {\tt #1}}

\newtheorem{note}[equation]{Note}
\newtheorem{example}[equation]{Example}
\title{Catching all geodesics of a manifold with moving balls and application to controllability of the wave equation}
\date{\today}
\author{Cyril LETROUIT}
\maketitle

\abstract{We address the problem of catching all speed $1$ geodesics of a Riemannian manifold with a moving ball: given a compact Riemannian manifold $(M,g)$ and small parameters $\e>0$ and $v>0$, is it possible to find $T>0$ and an absolutely continuous map $x:[0,T]\rightarrow M, t\mapsto x(t)$ satisfying $\|\dot{x}\|_{\infty}\leq v$ and such that any geodesic of $(M,g)$ traveled at speed $1$ meets the open ball $B_g(x(t),\e)\subset M$ within time $T$? Our main motivation comes from the control of the wave equation: our results show that the controllability of the wave equation can sometimes be improved by allowing the domain of control to move adequately, even very slowly. We first prove that, in any Riemannian manifold $(M,g)$ satisfying a geodesic recurrence condition (GRC), our problem has a positive answer for any $\varepsilon>0$ and $v>0$, and we give examples of Riemannian manifolds $(M,g)$ for which (GRC) is satisfied. Then, we build an explicit example of a domain $X\subset\R^2$ (with flat metric) containing convex obstacles, not satisfying (GRC), for which our problem has a negative answer if $\e$ and $v$ are small enough, i.e., no sufficiently small ball moving sufficiently slowly can catch all geodesics of $X$. 

}

\tableofcontents

\section{Introduction and main results} \label{intro}

There exist compact Riemannian manifolds $(M,g)$ for which, given any sufficiently small nonempty open subset of $M$, there exist geodesics which do not enter this set (e.g., if the manifold is the 2D sphere $\mathbb{S}^2$ or the 2D torus $\mathbb{T}^2$), or, said differently, no small nonempty open static set can ``catch'' all geodesics of the manifold. One may wonder whether, if we allow the subset to move according to an adequate path, even slowly, it would then become possible to ``catch'' all geodesics in some finite time $T>0$. Here, we assume that all geodesics travel at speed $1$. This is the problem we address in this paper. Besides its theoretical interest, this question also has an interesting application to the controllability of the wave equation, as we will explain in Section \ref{s:app}.

\smallskip

Let us consider $(M,g)$ a smooth connected compact Riemannian manifold of dimension $n\geq1$, with a smooth boundary if $\partial M\neq \emptyset$. 

\smallskip

When $\partial M\neq \emptyset$, the usual notions of geodesics and of bicharacteristics have to be generalized in order to take into account the reflections on $\partial M$. Generalized geodesics are usual geodesics inside $M$, and they reflect on $\partial M$ according to the laws of geometric optics. This generalization is called the generalized bicharacteristic flow of Melrose and Sj\"ostrand, see \cite{melrose1978singularities}. We do not recall the construction but we simply mention that, setting $Y=\mathbb{R}\times \overline{M}$, a generalized bicharacteristic $^{b}\gamma : \mathbb{R}\rightarrow \ ^{b}T^{*}Y$ is a continuous map which is uniquely determined if it has no point in $G^{\infty}$, the set of cotangent vectors with contact of infinite order with the boundary. Using $t$ as a parameter, generalized geodesics of $M$, traveling at speed $1$, are then the projection onto $M$ of generalized bicharacteristics.

\begin{remark}
If $\partial M$ is non-empty and not smooth, the generalized bicharacteristic flow is not necessarily well defined since there may be no uniqueness of a bicharacteristic at the points where $\partial M$ is not smooth. However, when uniqueness is ensured (for example in the case of conical singularities with right angles), our results are still true (see \cite[Remark 1.9]{le2017geometric} for more comments on this issue).
\end{remark}

We now recall the concept of time-dependent geometric control condition (in short, $t$-GCC, see \cite{le2017geometric}), which is particularly suited for our problem.
\begin{definition} \label{deftgcc}
Let $Q$ be an open subset of $\mathbb{R}\times \overline{M}$, and let $T>0$. We say that $(Q,T)$ satisfies the time-dependent geometric control condition (in short, $t$-GCC) if every generalized bicharacteristic $^{b}\gamma:\mathbb{R}\rightarrow \ ^{b}T^{*}Y$, $s\mapsto (t(s),y(s),\tau (s),\eta(s))$ is such that there exists $s\in\mathbb{R}$ such that $t(s)\in (0,T)$ and $(t(s),y(s))\in Q$. We say that $Q$ satisfies the the $t$-GCC if there exists $T>0$ such that $(Q,T)$ satisfies the $t$-GCC. 
\end{definition}

If there exists $T>0$ such that $(Q,T)$ satisfies $t$-GCC, the control time $T_{0}(Q)$ is defined by 
\[ T_{0}(Q)=\inf \{ T\in (0,+\infty) \  | \ (Q,T) \text{ satisfies the $t$-GCC}\}.\]

In this paper, the subsets $Q\subset \R\times \overline{M}$ which are considered are moving open balls.

\begin{definition} \label{defmoving}
Let $\varepsilon >0$ be a real number, $T>0$ be a fixed time and $x:[0,T]\rightarrow M$ be an absolutely continuous path. The moving domain $\omega(t)$ associated with the path $x(t)$ is defined for every $t\in [0,T]$ by $\omega(t)=B_g(x(t),\varepsilon)$, where $B_g(x(t),\varepsilon )$ is the open ball of center $x(t)$ and of radius $\varepsilon$ in $(M,g)$.
\end{definition}

In fact, we could consider more general moving domains, such as translations of a fixed shape (if $M\subset \R^n$), but for the sake of simplicity we keep Definition \ref{defmoving} for our moving domains.

\smallskip

Note also that we require all our paths $t\mapsto x(t)$ to be absolutely continuous because we need to define paths with bounded speed. 

\smallskip

Our problem then reads as follows:

\begin{quote}
\textbf{(P1):} Given small parameters $\e>0$ and $v>0$, is it possible to find $T>0$ and an absolutely continuous map $x:[0,T]\rightarrow M, t\mapsto x(t)$ satisfying $\|\dot{x}\|_{\infty}\leq v$ such that any generalized geodesic of $(M,g)$ traveled at speed $1$ meets the moving ball $\omega(t)=B_g(x(t),\e)\subset M$ within time $T$?
\end{quote}
Here $\|\dot{x}\|_\infty$ stands for the essential supremum over $[0,T]$ of $(g(\dot{x}(t),\dot{x}(t)))^{\frac12}$.

\paragraph{Role of the speed $v$.} In our results, the maximal speed $v$ plays a key role. Let us first remark that if $v$ is allowed to be very large, then it is easy to construct a moving domain $\omega(t)$ such that $t$-GCC is satisfied in time $T$ even with $T>0$ small.  For example, if we take a path $x(t)$ which, within short time (so that the geodesic rays do not have time to move much), passes near any point in $M$, then the associated moving domain $\omega (t)=B_g(x(t),\varepsilon )$ meets any geodesic ray. More precisely if $\varepsilon >0$ is fixed, and the absolutely continuous path $x:[0,\varepsilon /2]\rightarrow M$ is such that for any $x\in M$, there exists $t\in [0,\varepsilon /2]$ such that $|x(t)-x|<\varepsilon /4$, then the moving domain $\omega(t)=B_g(x(t),\varepsilon )$ meets any geodesic ray. Of course, in this case, the speed $|\dot{x}(t)|$ is very large (of the order of $\varepsilon^{-1}$) and is therefore not comparable with the speed of the geodesic rays (which is fixed to $1$).

\smallskip

Therefore, to make sense, positive answers to (P1) have to be established for a speed $v$ bounded independently of $\varepsilon$, or, even better, for any speed $v>0$. This is the case for example in Theorem \ref{thmovwave}.

\paragraph{Organization of the paper.} The paper is organized as follows. In Section \ref{intro}, we state the main results. Namely, in Section \ref{obsinfinitetime}, we provide a construction which shows that the answer to (P1) is positive for any $\e,v>0$ under a certain geodesic recurrence condition (GRC) on $(M,g)$. In Section \ref{examples} we give examples of manifolds $(M,g)$ satisfying (GRC). In Section \ref{s:neg}, we build an explicit example of domain $X\subset\R^2$ (with flat metric) containing convex obstacles, for which our problem has a negative answer if $\e$ and $v$ are small enough. In particular, this domain does not satisfy (GRC). In Section \ref{s:app}, we give an application of our results to the controllability of the wave equation. Finally, Section \ref{proofs} is devoted to the proof of our results.

\paragraph{Acknowledgment.} I warmly thank Emmanuel Tr\'elat for fruitful discussions and his careful reading of preliminary versions of this paper, and Yves Coud\`ene for giving a lot of information related to Theorem \ref{t:neg}. I also thank Romain Joly, Fran\c{c}ois Ledrappier and Marie-Claude Arnaud for interesting discussions.

\subsection{Main result} \label{obsinfinitetime}

Our first result roughly says that if for each geodesic trajectory $t\mapsto y(t)$ in $M$, there exists a small open (time-independent) set $U\subset M$ where the trajectory spends (asymptotically) some positive part of its time, then (P1) has a positive answer.

\begin{definition}Let $\varepsilon >0$ be a (small) positive real number. We say that the generalized geodesic $t\mapsto y(t)$ of $M$ satisfies the Geodesic Recurrence Condition (GRC) for $\varepsilon$ if
\begin{quote}
\textbf{(GRC)   } there exists an open ball $U\subset M$ (depending on $y(\cdot)$) of radius $<\varepsilon$ such that 
\begin{equation} \label{condliminf} \underset{T\rightarrow +\infty}{\liminf } \frac{|\{t\in [0,T]\mid y(t)\in U\}|}{T}>0.\end{equation}
\end{quote}
Here, the notation $|A|$ stands for the Lebesgue measure of a measurable subset $A\subset \R$.
\end{definition}

\begin{theorem}\label{thmovwave}
Let $\varepsilon >0$ be a positive real number such that any generalized geodesic $t\mapsto y(t)$ in $M$ satisfies (GRC) for $\varepsilon$. Let $v>0$ be a fixed speed. Then (P1) has a positive answer: there exist $T>0$ and an absolutely continuous path $x:[0,T]\rightarrow M$ satisfying $\|\dot{x}\|_\infty\leq v$ such that the moving domain $\omega(t)=B_g(x(t),\varepsilon) \cap M$ satisfies $t$-GCC in time $T$. 
\end{theorem}

This answers to a problem raised in \cite[Section 3E]{le2017geometric}. We do not know of other references in the literature introducing condition (GRC). Determining the manifolds $(M,g)$ all of whose geodesics satisfy (GRC) seems to be a difficult question in general. In Section \ref{examples} we give some examples of manifolds $(M,g)$ for which (GRC) is satisfied.

\begin{remark} \label{countable}
Note that a given geodesic $t\mapsto y(t)$ satisfies (GRC) if and only if for any $s\in\R$, the geodesic $t\mapsto y(t+s)$ satisfies (GRC): (GRC) depends only on the trace of the geodesic since it is invariant by translations in time. We note that our proof of Theorem \ref{thmovwave} can easily be extended to the case where all traces of geodesics but a countable number satisfy (GRC). For the sake of simplicity, we did not include this obvious extension in the statement of the theorem. 
\end{remark}

\subsection{Examples of manifolds $M$ satisfying (GRC)} \label{examples}
In this section, we give some examples of smooth connected compact Riemannian manifolds $(M,g)$ for which all geodesics satisfy (GRC), so that Theorem \ref{thmovwave} applies. Let us first recall the definition of the dichotomy property (see \cite{smillie2000dynamics}).

\begin{definition}
A bounded manifold $(M,g)$ satisfies the dichotomy property if each of its geodesics is either periodic or uniformly distributed, the latter meaning that for every open set $U \subset M$, 
\begin{equation*} \underset{T\rightarrow +\infty}{\lim} \frac{|\{t\in [0,T] \mid y(t)\in U\}|}{T}=\frac{\vol_{g}(U)}{\vol_{g}(M)}.\end{equation*}
\end{definition}

Typical examples are the square and the rectangles with the flat metric. More generally, any polygon that tiles the plane by reflection has the dichotomy property. In fact, this property is satisfied by all ``lattice examples" (see \cite{smillie2000dynamics}).

\begin{proposition} \label{dicho}
If $(M,g)$ satisfies the dichotomy property, then all its geodesics satisfy (GRC).
\end{proposition}

If $(M,g)$ satisfies the weaker property that each geodesic is either periodic or uniformly distributed in some open subset $M' \subset M$, the same proof shows that each geodesic of $M$ also satisfies (GRC), so that Theorem \ref{thmovwave} also applies. With the same argument, we can prove the following proposition.

\begin{proposition} \label{disk} Any geodesic of the two-dimensional disk satisfies (GRC) but the two-dimensional disk with flat metric does not satisfy the dichotomy property. \end{proposition}

Let us also give a property which is somewhat weaker than (GRC) (although not exactly because the parameter $\varepsilon$ is fixed in (GRC) and arbitrary in Proposition~\ref{weaker}) but which is satisfied for any Riemannian manifold $(M,g)$.

\begin{proposition} \label{weaker}
For any geodesic $t\mapsto y(t)$ and any $\varepsilon >0$, there exists an open ball $B\subset M$ of radius $\varepsilon$ and an increasing sequence of times $(T_{n})_{n\in\mathbb{N}^{*}}$ tending to $+\infty$ such that
\begin{equation*} \underset{n\rightarrow +\infty}{\liminf} \frac{|\{t\in [0,T_{n}] \mid y(t)\in B\cap M\}|}{T_{n}}>0. \end{equation*}
\end{proposition}

This proposition roughly means that (GRC) is verified up to a subsequence. Although being very natural, it is of no help for providing positive answers to (P1). In some sense, in order to give a positive answer to (P1), we need a ``quantitative'' version of Proposition 3 to hold, where no extraction of a sequence of times is needed. This ``quantitative version'' is provided here by the condition (GRC). Without this condition, it is hard to design the trajectory of the moving domain capturing all speed $1$ geodesics (the proof of Theorem \ref{thmovwave} fails), and sometimes it is even impossible as proved in Theorem \ref{t:neg}.

\subsection{Negative answer for a domain with convex obstacles}\label{s:neg}

It is not true that, for any smooth connected compact Riemannian manifold $(M,g)$ and any parameters $\e,v>0$, problem (P1) has a positive answer: in some domains with convex obstacles, the dynamics is so chaotic that (P1) has a negative answer,  i.e., no sufficiently small ball moving sufficiently slowly can catch all geodesics. Note that geodesics in domains with convex obstacles have already been studied in detail (see for example \cite{morita1991symbolic}).

\begin{theorem} \label{t:neg}
There exists a bounded open subset $X$ of $\mathbb{R}^{2}$ with smooth boundary $\partial X$ (and equipped with the flat metric of $\R^2$) such that for sufficiently small $\e,v>0$, (P1) has a negative answer: for any $T>0$ and any absolutely continuous path $x:[0,T]\rightarrow X$ with $\|\dot{x}\|_\infty\leq v$, the moving domain $\omega(t)$ defined by $\omega(t)=B(x(t),\varepsilon) \cap X$ does not satisfy the $t$-GCC in time $T$. 
\end{theorem}

Note that the question to determine in full generality the answer to problem (P1), i.e., for an arbitrary smooth connected compact Riemannian manifold $(M,g)$ and for arbitrary parameters $\e,v>0$, is open.

\subsection{Application to the control of the wave equation} \label{s:app}
In this section, we present an application of our results to the control of the wave equation. We first recall a few basic facts on controllability of the wave equation.
\paragraph{Controllability of the wave equation.} The study of controllability properties for the wave equation goes back at least to the work of Russell \cite{russell1971boundary}, \cite{russell1971boundary2}. By exact controllability in time $T>0$ for the wave equation
\begin{equation} \label{waveeq} \partial^{2}_{tt}u-\triangle_g u=\chi_{\omega}f\end{equation}
on a smooth connected compact Riemannian manifold $(M,g)$ with or without boundary controlled in an open subset $\omega\subset M$, we mean that given an initial state $(u_{0},u_{1})$ and a final state $(u_{0}^{F}, u_{1}^{F})$, it is possible to find a control $f$ such that the solution of (\ref{waveeq}) with initial datum $(u_{t=0},\partial_{t}u_{t=0})=(u_{0},u_{1})$ verifies $(u_{t=T},\partial_{t}u_{t=T})=(u_{0}^{F},u_{1}^{F})$. In case of Dirichlet boundary conditions, we take the initial datum $(u_{0},u_{1})$ in the energy space $H_{0}^{1}(M)\times L^{2}(M)$ and we seek $f\in L^{2}((0,T)\times M)$. 

\smallskip

By duality, exact controllability of the wave equation is equivalent to the observability  inequality 
\begin{equation} \label{obsineq}
C(\|u_{|t=0}\|^{2}_{H^{1}(M )}+\|\partial_{t}u_{|t=0}\|^{2}_{L^{2}(M)})\leq \int_{0}^{T} \|\partial_{t}u\|^{2}_{L^{2}(\omega)}dt
\end{equation}
for any solution $u$ of the free wave equation $\partial^{2}_{tt}u-\triangle_g u=0$ where $C>0$ does not depend on $u$ (\cite{lions1988controlabilite}). This last inequality means that it is possible, from the observation of $u$ in the region $\omega$, to recover $u$ in the whole manifold $M$, with an accuracy which is measured by the best possible constant $C$ such that (\ref{obsineq}) is satisfied.

\smallskip

If the open set $\omega$ satisfies the geometric control condition (GCC) in time $T$, which roughly means that all rays of geometric optics in $M$ meet $\omega$ within time $T$, then the results of \cite{rauch1974exponential} and \cite{bardos1992sharp} show that the infimum of all possible times $T$ such that (\ref{obsineq}) is verified coincides with the infimum of the times $T$ such that the geometric control condition is verified in time $T$. The geometric control condition is ``almost" necessary and sufficient (see \cite{humbert2019observability}). The ``almost'' is due to grazing rays, i.e., rays of geometric optics which touch $\partial \omega$ but do not enter $\omega$.

\smallskip

 It is natural to generalize GCC to a time-dependent setting, i.e. for a domain of observation $\omega$ that is allowed to move in time. In other words, the domain of observation is now a measurable subset $Q$ of $(0,T)\times \overline{M}$, which is not necessarily a cylinder $(0,T)\times \omega$ as in the time-independent setting. As will be stated below in a precise manner, if $(Q,T)$ satisfies $t$-GCC, then the wave equation is observable in $Q$.

\smallskip

The search for time-varying observation domains is motivated for instance by seismic exploration, in order to address situations in which all sensors cannot be active at the same time. In many practical examples, it is also possible to move sensors in order to get better precision in the inverse problems which arise while trying to recover data. 

\smallskip

There is not much literature about observation of the wave equation on moving domains. The first paper to address this question (in one dimension of space) seems to be \cite{khapalov1995controllability}. More recently, in \cite{le2017geometric} the authors have proved that the $t$-GCC condition is an almost necessary and sufficient condition for controllability of $n$-dimensional waves by a moving domain.  The question we address in this paper was raised as an open problem in \cite[Section 3E]{le2017geometric}. In \cite{castro2014controllability} the authors proved the same result as in \cite{le2017geometric} for the one-dimensional wave equation, and then characterized the minimal norm controls. Finally, the paper \cite{castro2013exact} gives sufficient conditions on the trajectory of a moving interior point of an interval to ensure controllability of the wave equation.

\smallskip

We now give precise definitions and recall the result which will be used in the sequel.

\paragraph{Setting.} \label{setting} We adopt the same setting as in \cite{le2017geometric}. We recall it here for the sake of completeness. Let $(M,g)$ be a smooth connected compact $n$-dimensional Riemannian manifold with $n\geq 1$. We consider the wave equation 
\begin{equation} \label{eqondes}
\partial_{tt}^{2}u-\triangle_{g}u=0
\end{equation}
in $\mathbb{R}\times M$, where $\triangle_{g}$ denotes the Laplace-Beltrami operator on $(M,g)$. If the boundary $\partial M$ of $M$ is nonempty, then we consider boundary conditions of the form
\begin{equation} \label{condbord}
Bu=0 \ \ \ \text{on $\mathbb{R}\times\partial M$}
\end{equation}
where the operator $B$ is either
\begin{itemize}
\item the Dirichlet trace operator, $Bu=u_{|\partial M}$;
\item or the Neumann trace operator, $Bu=\partial_{n}u_{|\partial M}$, where $\partial_{n}$ is the outward normal derivative along $\partial M$.
\end{itemize}

In the case of a manifold without boundary or in the case of homogeneous Neumann boundary conditions, the Laplace-Beltrami operator is not invertible on $L^{2}(M )$ but is invertible in 
\[ L^{2}_{0}(M)=\left\{ u\in L^{2}(M ) \ \big{|} \ \int_{M}u(x)dx_{g}=0\right\}.\]

In what follows, we set $\mathcal{H}=L_{0}^{2}(M)$ in the boundaryless case or in the Neumann case, and $\mathcal{H}=L^{2}(M)$ in the Dirichlet case (in both cases, the norm on $\mathcal{H}$ is the usual $L^{2}$-norm). We denote by $A=-\triangle_{g}$ the operator defined on the domain
\[ D(A)=\{ u\in \mathcal{H} \ | \ Au\in X \text{ and } Bu=0 \}\]
with one of the above boundary conditions whenever $\partial M\neq \emptyset$. We refer to \cite{le2017geometric} for an explicit description of $D(A)$, $D(A^{1/2})$ and of $D(A^{1/2})'$. 

\smallskip

For all $(u^{0},u^{1})\in D(A^{1/2})\times \mathcal{H}$, there exists a unique solution $u\in C^{0}(\mathbb{R} ; D(A^{1/2}))\cap C^1(\mathbb{R};\mathcal{H})$ of (\ref{eqondes})-(\ref{condbord}) such that $u_{|t=0}=u^{0}$ and $\partial_{t}u_{|t=0}=u^{1}$. Such solutions of (\ref{eqondes})-(\ref{condbord}) are understood in the weak sense. 

\smallskip

Let $Q$ be an open subset of $\mathbb{R}\times\overline{M}$. We set 
\[ \omega(t)=\{ x\in \overline{M} \ | \ (t,x)\in Q\}.\]
Let $T>0$ be arbitrary. We say that (\ref{eqondes})-(\ref{condbord}) is observable on $Q$ in time $T$ if there exists $C>0$ such that 
\begin{equation}\label{obsineq3}
C\|(u_{|t=0},\partial_{t}u_{|t=0})\|^{2}_{D(A^{1/2})\times \mathcal{H}}\leq \int_{0}^{T}\int_{\omega(t)}|\partial_{t}u(t,x)|^{2}dx_{g}dt
\end{equation}
for any solution $u$ of (\ref{eqondes})-(\ref{condbord}).

\smallskip

The main theorem of \cite{le2017geometric} states:
\begin{quote}
{\it Let $Q$ be an open subset of $\mathbb{R}\times \overline{M}$ that satisfies the $t$-GCC. Let $T \in (T_{0}(Q),+\infty)$. When $\partial M \neq \emptyset$, we assume moreover that no generalized bicharacteristic has a contact of infinite order with $(0,T)\times \partial M$, that is, $G^{\infty}=\emptyset$. Then the observability inequality (\ref{obsineq3}) holds, and therefore the wave equation is controllable in time $T$ with control domain $Q$.}
\end{quote}

As a corollary of Theorem \ref{thmovwave}, we have:
\begin{theorem} \label{t:app}
If $(M,g)$ satisfies (GRC), then there for any $\e>0$ and $v>0$, there exist $T>0$ and an absolutely continuous path $x:[0,T]\rightarrow M, t\mapsto x(t)$ satisfying $\|\dot{x}\|_\infty\leq v$ such that the wave equation in $M$ is controllable in time $T$ with moving control domain $\omega(t)=B_g(x(t),\e)$.
\end{theorem}

Let us describe shortly an application of Theorem \ref{t:app} to the 2D torus $M = \mathbb{T}^{2}$ with the flat metric on it. It is an example of a manifold such that the wave equation is not observable on any static ball of sufficiently small radius $\varepsilon$: for any sufficiently small ball, there is a periodic geodesic which does not meet this ball, and thus GCC is not verified. Constructing Gaussian beams along this geodesic (\cite{ralston1982gaussian}), we see that the wave equation is not observable on such a small ball. Since $\mathbb{T}^2$ satisfies (GRC) (by application of Proposition \ref{dicho}), Theorem \ref{t:app} however guarantees that if we allow the small ball to move (even with an arbitrarily small speed) in $\mathbb{T}^{2}$, it is always possible to make the wave equation observable on this moving domain.

\smallskip

In the proof of Theorem \ref{t:neg}, we show that for any small ball moving slowly over a time interval $[0,T]$, there exists a generalized geodesic which stays at positive distance of this moving ball during the time interval $[0,T]$. Therefore, using the construction of Gaussian beams along this geodesic (\cite{ralston1982gaussian}), we get:
\begin{theorem} \label{t:noobs}
On the domain $X$ (with the flat metric of $\R^2$) of Theorem \ref{t:neg}, for any $\e>0$, $v>0$, $T>0$ and any absolutely continuous path $x:[0,T]\rightarrow M, t\mapsto x(t)$ such that $\|\dot{x}\|_\infty\leq v$, the wave equation in $M$ is not controllable in time $T$ with control domain $\omega(t)=B(x(t),\e)$.
\end{theorem}

In relation with the domain $X$ exhibited in Theorem \ref{t:neg}, note that the wave equation in manifolds with convex obstacles has already been studied, although not for the same purpose. See for example \cite{ikawa1982decay}, where the trapped geodesics are used to establish some decay estimates for the energy of the free wave equation, \cite{burq2004geometric} for resolvent estimates, and \cite{joly2019decay} for a recent use in the context of the damped wave equation.

\section{Proofs} \label{proofs}

\subsection{Proof of Theorem \ref{thmovwave}}
Let $\varepsilon >0$ and $v>0$. We consider a dense sequence of points $(x_{i})_{i\in\mathbb{N}^{*}}$ in $M$, we set $B_{i}=B_g(x_{i},\varepsilon )$ for $i\in \mathbb{N}^{*}$. Then we construct $x:\mathbb{R}^{+}\rightarrow M$ a map of speed $|\dot{x}(t)|\leq v$ in the following way. The point $x(t)$ will successively stay on the points $x_{i}$ in the following order (we will precise the time it stays on each point later): \begin{displaymath} x_{1}, x_{2}, x_{1}, x_{2}, x_{3}, x_{4}, x_{1}, x_{2}, ..., x_{8}, x_{1}, x_{2}, ..., x_{16}, ..., x_{1} , ..., x_{2^n}, x_{1}, ...\end{displaymath} and this sequence continues until infinity. We call each of these positions a ``step": for example, at step 1, $x(t)=x_{1}$, at step 2, $x(t)=x_{2}$, at step 3, $x(t)=x_{1}$, etc. Of course, between two steps, there is a smooth transition: the point $x(t)$ goes from one $x_{i}$ to the following smoothly. Then we have to specify how much time $x(t)$ stays on each $x_{j}$ (i.e., the time duration of each step): we require that if $x(t)$ arrives at step $j$ at time $t$, then step $j$ lasts $t2^{j}$ seconds (much more time than all the time already passed). Lastly, we take $\omega(t)=B_g(x(t),\varepsilon)$. This construction implies that for each $i\in\mathbb{N}^{*}$, the following assertion is true:
\begin{quote}\textbf{(B) :} For every constant $0<K<1$ and every $T>0$, there exist $T''>T' >T$ such that $\frac{T''-T'}{T''}>K$ and $x(t)=x_{i}$ for $t\in [T',T'']$.\end{quote}
Now consider a geodesic $t\mapsto y(t)$ in $M$. We will show that there exists $t\geq 0$ such that $y(t)\in \omega(t)$. Let $U$ be an open ball of radius $<\e$ satisfying (GRC) for the trajectory $t\mapsto y(t)$. Since $(x_{i})_{i\in\mathbb{N}^{*}}$ is dense in $M$, there exists $j\in \mathbb{N}^{*}$ such that $U \subset B_{j}=B_g(x_{j},\varepsilon )$ with the above notations. Note that (GRC) implies that 
\begin{equation*}  \underset{T\rightarrow +\infty}{\text{liminf }} \frac{|\{t\in [0,T] \mid y(t)\in B_{j}\}|}{T}\geq 3C >0\end{equation*}
for some $C>0$. 
This means that the trajectory $y(t)$ spends at least a fraction $3C$ of time in $B_{j}$ when time goes to infinity. Let $T_1$ be such that
\begin{equation} \label{fractionoftime} \forall t\geq T_1, \quad  \frac{|\{s\in [0,t] \mid y(s)\in B_{j}\}|}{t}\geq 2C. \end{equation}
By assertion (B) for $K=1-C$, there exist $T''>T'>T_1$ such that 
\begin{equation} \label{1er} \text{$(T''-T')/T''>1-C$ \quad and \quad $x(t)=x_{j}$ \quad for $t\in [T',T'']$.} \end{equation}
 If we take $t=T''$ in (\ref{fractionoftime}), we get
\begin{equation} \label{2eme} \frac{|\{s\in [0,T''] \mid y(s)\in B_{j}=B_g(x_{j},\varepsilon )\}|}{T''}\geq 2C.\end{equation}
Combining (\ref{1er}) et (\ref{2eme}) we see that there exists a time $t\leq T''$ such that $y(t)\in \omega(t)$ (and at this time $t$, $\omega(t)=B_j$). Therefore, any trajectory meets $\omega(t)$ when time goes to infinity. 

\smallskip

The function which associates to a point in $S^*M$ (and, by extension, to the bicharacteristic $(y(t),\eta(t))$ starting from this point at time $0$) the infimum of all $t\in \R^+$ for which $y(t)\in\omega(t)$ is upper semi-continuous. Since $S^*M$ is compact, it has a maximum. It follows that there exists $0<T<+\infty$ such that $\omega(t)$ satisfies the $t$-GCC in time $T$.

\subsection{Proof of Proposition \ref{dicho}}
We first note that all closed geodesics in $M$ satisfy (GRC). To see it, just fix a periodic geodesic $t\mapsto y(t)$, let $T_{1}>0$ be its minimal period and set $U=B(y(0),\varepsilon /2)\cap M$. Then the liminf appearing in (\ref{condliminf}) is greater than or equal to $1/T_{1}$ and hence is positive. Secondly, for the uniformly distributed geodesics, any open set $U\subset M$ can be used in (GRC). This concludes the proof of Proposition \ref{dicho}.

\subsection{Proof of Proposition \ref{disk}}
Let $D$ be an open two-dimensional disk. If we take a sufficiently short chord of the disk intercepting an angle $\alpha$ which is not commensurable to $\pi$, then the geodesic which follows this chord is not periodic (since $\frac{\alpha}{\pi}\notin \mathbb{Q}$) and it is not uniformly distributed (since it does not meet a small disk with the same center as $D$ if the chord is sufficiently short). 

\smallskip

However, it is possible to verify that any geodesic of $D$ is either periodic or there exists $0<r<1$ (depending on the geodesic) such that any open subset $U$ of the annulus $D\backslash rD$ satisfies (GRC). To see it, take a geodesic $t\mapsto y(t)$, $t\geq 0$ in $D$ which is not periodic. Its trace on $\partial D$ is a sequence of points $x_{1},x_{2},...$ which are always separated by the same distance. If one sees $\partial D$ as the quotient $\mathbb{R}/\pi\mathbb{Z}$, then the points $x_{i}$ form an arithmetic sequence which is dense in $\mathbb{R}/\pi\mathbb{Z}$. Moreover, since the oriented angle $\alpha$ made by the forward trajectory $t\mapsto y(t)$ at each $x_{i}$ with the tangent to $\partial D$ is always the same, it is straightforward to see that there exists $0<r<1$ such that $M '=\overline{D}\backslash rD$ is the closure of the trajectory $t\mapsto y(t)$. This oriented angle $\alpha$ also determines, for each point $x\in D$, a unique point $\pi_{\alpha}(x)$ on $\partial D$ which is the only point for which the segment with ends $\pi_{\alpha}(x)$ and $x$ makes an oriented angle $\alpha$ with the tangent at $\pi_{\alpha}(x)$. Let $U$ be a small ball in $M '$ and let $\pi_{\alpha}^{-1}(U)$ be its preimage by $\pi_{\alpha}$. It is an interval on the boundary $\partial D$. It means that if $x_{i}\in \pi_{\alpha}^{-1}(U)$, the trajectory $t\mapsto y(t)$ falls in $U$ while running from $x_{i}$ to $x_{i+1}$. Proposition \ref{disk} then follows from the uniform distribution of the sequence $(x_{i})_{i\in\mathbb{N}^*}$ in $\partial D$, since $\frac{\alpha}{\pi}\notin \mathbb{Q}$.

\subsection{Proof of Proposition \ref{weaker}}
Let $t\mapsto y(t)$ be a (generalized) geodesic in $M$ (or in $\overline{M}$ if $M$ has a boundary). For the sake of simplicity, we assume in what follows that $M$ has no boundary, but the proof also works in the case with boundary. We set, for $n\in\mathbb{N}$, 
\[ \mu_{n}=\frac{1}{n}\sum_{i=1}^{n}\delta_{y(i)}\]
where $\delta_{x}$ is the Dirac measure in $M$ located at $x$. The measure $\mu_{n}$ is a probability measure on $M$. By Prokhorov's theorem, there exists a subsequence $(n_{k})_{k\in\mathbb{N}}$ and a probability measure $\mu$ on $M$ such that $\mu_{n_{k}}\rightharpoonup \mu$ in the weak-* topology of measures. There exists $M_{1}$ an open ball of radius $\varepsilon/4$ such that $\mu (M_{1})>0$. We denote by $M_{2}$ the ball of radius $\varepsilon/2$ with the same center as $M_{1}$. Let $f:M\rightarrow [0,1]$ be a continuous function equal to $1$ on $M_{1}$ and to $0$ outside of $M_{2}$. Then
\[ \int_{M}fd\mu_{n_{k}}\rightarrow \int_{M}fd\mu >0\]
as $k\rightarrow +\infty$. But $ \mu_{n_{k}}(M_{2})\geq \int_{M_{2}}fd\mu_{n_{k}}=  \int_{M}fd\mu_{n_{k}}$ since $f$ vanishes outside of $M_{2}$. Therefore
\begin{equation}\label{liminf2}\underset{k\rightarrow +\infty}{\liminf} \ \mu_{n_{k}}(M_{2}) >0.\end{equation}
Note that if $y(i)\in M_{2}$, then $y(t)\in B$ for any $t$ satisfying $|t-i|\leq \varepsilon/2$, where $B$ denotes the ball with the same center as $M_{1}$ and $M_{2}$ and with radius $\e$. Combining this remark with (\ref{liminf2}), we get 
\begin{equation*} \liminf_{k\rightarrow +\infty}   \frac{|\{t\in [0,n_{k}] \mid y(t)\in B\}|}{n_{k}}>0,\end{equation*}
which finishes the proof of Proposition \ref{weaker}.

\subsection{Proof of Theorem \ref{t:neg}}

We consider three circles $\mathcal{C}^{1}, \mathcal{C}^{2}$ and $\mathcal{C}^{3}$ in the Euclidean plane, with same radius $r_0>0$ (with $r_0\ll 1$), whose centers form an equilateral triangle $\Delta$ of side-length $1+2r_0$. These circles will be part of the boundary $\partial X$ of the domain $X$. More precisely, we set $\mathcal{C}=\mathcal{C}^{1}\cup \mathcal{C}^{2}\cup\mathcal{C}^{3}$ and we take $X$ to be a domain with smooth boundary $\partial X=\mathcal{C}\cup \Gamma$ where $\Gamma$ is a smooth curve which encloses the convex hull of the three circles. The set $X$ is therefore connected but $\partial X$ has four connected components. The precise form of $\Gamma$ does not matter since the geodesic trajectories we will construct only bounce on the obstacles $\mathcal{C}^{1}$, $ \mathcal{C}^{2}$ and $\mathcal{C}^{3}$. Our notations are summarized in Figure \ref{f:domain}.

\smallskip

Note that with our choice of side-length for $\Delta$, between any two of the three circles $\mathcal{C}^{1},\mathcal{C}^{2},\mathcal{C}^{3}$, there is a unique trajectory of length $1$ joining these two circles (which is also the trace of the periodic trajectory bouncing between the two circles).

\begin{figure}[!h] 
\begin{center}
\includegraphics[width=9cm]{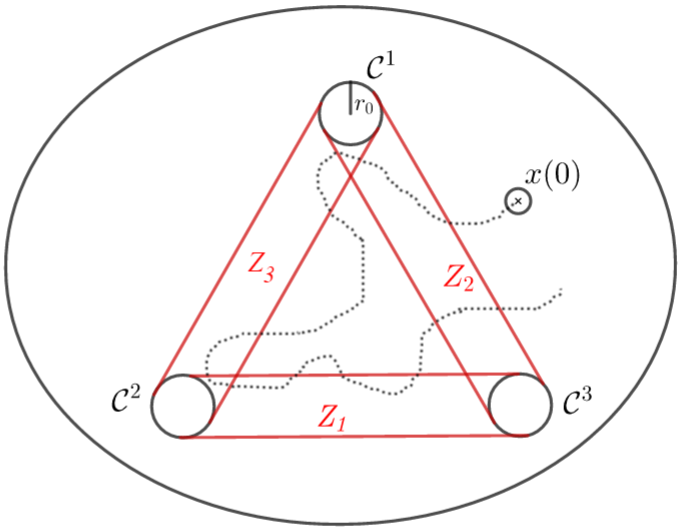}
\caption{\textit{The domain $X$. The trajectory $t\mapsto x(t)$ is the dashed curve.}} 
\label{f:domain}
\end{center}
\end{figure}

We will show that for $X$ as above with $r_0$ sufficiently small, there exist $\e,v>0$ such that for any $T>0$ and any absolutely continuous path $x:[0,T]\rightarrow X$ with $\|\dot{x}\|_\infty\leq v$, the moving domain $\omega(t)$ defined by $\omega(t)=B(x(t),\varepsilon) \cap X$ does not satisfy the $t$-GCC in time $T$. 

\smallskip

Therefore, we fix small enough positive parameters $r_0,v,\e\ll1$, a time $T>0$ and an absolutely continuous path $x:[0,T]\rightarrow X$ with $\|\dot{x}\|_\infty\leq v$. Finally we set $\omega(t)=B(x(t),\varepsilon) \cap X$. Our goal is to construct a generalized geodesic $\gamma$ of $X$ (with the flat metric)  such that for any $t\in [0,T]$, $\gamma(t)\notin \omega(t)$.

\smallskip

The geodesic flow in $X$ is chaotic. For example, the following fact is well known:

\begin{fact} \label{coding}
For any sequence $(\xi_{n})\in \{1,2,3\}^{\mathbb{N}}$ such that $\xi_j\neq \xi_{j+1}$ for any $j\in\mathbb{N}$, there exists a geodesic which bounces successively on circles $\mathcal{C}^{\xi_{0}},\mathcal{C}^{\xi_{1}},\ldots,\mathcal{C}^{\xi_{n}},\ldots$ 
\end{fact}

A proof is given for example in \cite[Theorem 0]{morita1991symbolic}, and we also recall in Appendix \ref{a:fact} an elementary proof which explicitly shows the Cantor structure underlying this billiard flow.

\smallskip

Here we use Fact \ref{coding} to prescribe the behaviour of a geodesic not by giving the successive obstacles on which it bounces but by prescribing at any time $t\in [0,T]$ the ``\textit{approximate}'' zone of $X$ where the geodesic is at time $t$. Our proof of Theorem \ref{t:neg} relies on Fact \ref{coding}, but using only this fact without any mention to the position of $\gamma$ at any time would not be sufficient for our purpose: a priori, the statement of Fact \ref{coding} gives no indication about the times at which the geodesic hits the obstacles, but this is precisely the information that we need for the proof of Theorem \ref{t:neg}. Fact \ref{coding} together with the stability estimate of \cite[Theorem 0]{morita1991symbolic} will be the key ingredients of our proof.

\smallskip

Let us give a sketch of the proof:
\begin{itemize}
\item We define three zones $Z_1, Z_2, Z_3$, which are the pairwise convex hulls of the three circles $\mathcal{C}^1, \mathcal{C}^2, \mathcal{C}^3$ (see Figures \ref{f:domain} and \ref{f:split}). Note that any geodesic of $X$ which bounces only on $\mathcal{C}^1,\mathcal{C}^2$ and $\mathcal{C}^3$ (i.e., which never touches $\Gamma$) remains in $Z_1\cup Z_2\cup Z_3$.
\item In order construct a geodesic $\gamma$ never meeting $\omega$, we associate to each time $t\in [0,T]$ the indices of the zones $Z_1, Z_2, Z_3$ intersecting $\omega(t)$ (which are ``prohibited zones for $\gamma(t)$). This gives a function $f:[0,T]\rightarrow \mathcal{P}(\{1,2,3\})$ (the set of all subsets of $\{1,2,3\}$) and our goal is to construct $\gamma$ such that for any $t\in [0,T]$, $\gamma(t)\notin  \bigcup_{a\in f(t)} Z_{a}$. This means that the interval $[0,T]$ is split as in Figure \ref{f:split}, where the indices of the prohibited zones are written above the segment $[0,T]$. 
\item if it were possible, it would be sufficient to prescribe the times at which $\gamma$ passes from a zone to another, while avoiding the prohibited zones given by $f$ (just prescribing a certain number of bounces in each zone). But it is not easy and maybe even not feasible because any two successive bounces on the obstacles $\mathcal{C}^1, \mathcal{C}^2, \mathcal{C}^3$ are separated by a time almost equal to $1$ but not exactly $1$, and this small discrepancy heavily depends on the next zones that $\gamma$ will visit. If we prescribe the successive number of bounces of $\gamma$ in each zone, small discrepancies may accumulate, and it then becomes clear that we have to choose the number of bounces in each zone so as to compensate these accumulated discrepancies. A way to do this is to use the stability estimate \cite[Theorem 0]{morita1991symbolic}.
\item Let us consider again Figure \ref{f:split}. First, we construct a sequence of times (and also a sequence of associated zones) at which we would like ``ideally'' for $\gamma$ to switch from one zone to another: they are denoted by $0=T_0<T_1<\ldots<T_{n-1}<T$ in Lemma \ref{l:times} and in Figure \ref{f:split}. The difficulty is that there does not necessarily exist any geodesic $\gamma$ lying in the prescribed zone on any time interval $[T_j,T_{j+1}]$. Now, we note that since $\|\dot{x}\|_\infty\leq v$, it takes a time greater than $1/(2v)$ (which is large since $v\ll1$) for $\omega(t)$ to visit all zones $Z_1, Z_2$ and $Z_3$. Therefore, these switching times may be chosen quite flexibly. Thanks to this flexibility, we construct a geodesic $\gamma$ whose ``switching times'' $0=T_0'<T_1'<\ldots<T_n'=T$ are a perturbation of the sequence $(T_j)_{0\leq j\leq n-1}$ (i.e., $|T_j'-T_j|$ is small for any $j$). By construction, $\gamma$ does not meet $\omega(t)$ during the time interval $[0,T]$.
\end{itemize}

\begin{figure}[!h] 
\begin{center}
\includegraphics[width=14cm]{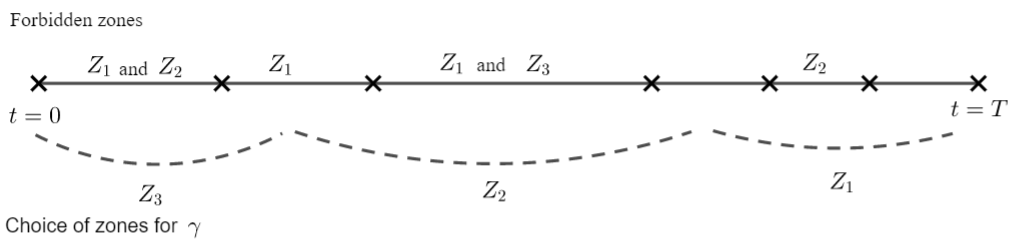} \caption{\textit{The interval $[0,T]$}} 
\label{f:split}
\end{center}
\end{figure}

We now give the full proof of Theorem \ref{t:neg}.

\smallskip

Let us introduce a few notations borrowed from \cite{morita1991symbolic} (we keep these notations in order to facilitate reading). We denote by $SX=X\times \mathbb{S}^1$ the unit tangent bundle of $X$ and by $\pi:SX\rightarrow X$ the canonical projection. We denote by $S_t$ the billiard flow on $SX$. We choose a point $q_j\in \mathcal{C}^j$ for $j=1,2,3$ and we define the following quantities for $x=(q,v)$ with $q\in \mathcal{C}= \mathcal{C}^1\cup  \mathcal{C}^2\cup \mathcal{C}^3$:
\begin{align} 
\xi_0(x)=&\ j \text{ if } q\in \mathcal{C}^j; \nonumber \\
r(x)=&\text{ the arclength between } q_{\xi_0(x)} \text{ and } q, \nonumber \\
&\text{ measured clockwise along the curve } \mathcal{C}^{\xi_0(x)}; \nonumber \\
\phi(x)= &\text{ the angle between the vector $v$ and the unit innernormal} \nonumber \\
&\text{ $n(q)$ of $\mathcal{C}^{\xi_0(x)}$ ar $q$, measured unitclockwise.} \nonumber
\end{align}
Note that the variable $x$ of these notations should not be confused with the notation for the trajectory $t\mapsto x(t)$ (we used $x$ to keep the same notations as  \cite{morita1991symbolic}). Similarly, the notation $M_-$ below has nothing to do with the notation $M$ of the previous sections. 

\smallskip

Then $\pi^{-1}(\mathcal{C})$ is parametrized as
\begin{equation*}
\pi^{-1}(\mathcal{C})=\{(j,r,\phi) \mid \ 1\leq j\leq 3, \ 0\leq r \leq 2\pi r_0, \ 0\leq \phi<2\pi\}.
\end{equation*}
We set 
\begin{equation*}
M_-=\left\{x\in \pi^{-1}(\mathcal{C}) \mid \ \frac{\pi}{2}\leq \phi(x)\leq \frac{3}{2}\pi\right\}
\end{equation*}
the set of incidental vectors and we introduce the equivalence relation
\begin{equation*}
x\sim y \text{\quad if and only if \quad} \Inv(x)=y \text{ or } x=y
\end{equation*}
where $\Inv:\pi^{-1}(\mathcal{C})\rightarrow \pi^{-1}(\mathcal{C})$ is defined by 
\begin{equation*}
\Inv(j,r,\phi)=(j,r,\pi-\phi)  \text{ mod } 2\pi.
\end{equation*}
It is then natural to identify the quotient $\pi^{-1}(\mathcal{C})/\sim$ with $M_-$, and we often use this identification in the sequel.

\smallskip

We define the first collision time $\tau_+$ and the last collision time $\tau_{-}$ by 
\begin{align}
\tau_+(x)&=\inf\{t>0 \mid \ \pi(S_tx)\in \mathcal{C}\} \nonumber \\
\tau_-(x)&=\sup\{t<0 \mid \ \pi(S_tx)\in \mathcal{C}\} \nonumber
\end{align}
with the convention that $\tau_+(x)$ (resp. $\tau_-(x)$) is $+\infty$ (resp. $-\infty$) if the set in the definition is empty.

\smallskip

The nonwandering set of the flow $S_t$ is
\begin{equation*}
\Omega=\{x\in \pi^{-1}(X)\cup M_- \mid \ \forall t\in\R, \ \pi(S_tx)\notin\Gamma\}.
\end{equation*}
We also define $\Omega_-=M_-\cap \Omega$. We define the local map $T$ and $T^{-1}$ by 
\begin{align}
T(x)&=S_{\tau_+(x)}(x) \  \text{ if } \tau_+(x)<+\infty, \nonumber \\
T^{-1}(x)&=S_{\tau_-(x)}(x) \ \text{ if } \tau_-(x)>-\infty, \nonumber
\end{align} 
and we see that $T^{-1}$ is the inverse map of $T$ and $T$ is locally diffeomorphic.

\smallskip

For $x\in \Omega_-$, we put $\xi_j(x)=\xi_0(T^jx)$. The sequence $\xi=(\xi_j)_{j=-\infty}^{\infty}$ is called the itinerary of $x$ if $\xi_j=\xi_j(x)$ and we write $\xi$ as $\xi(x)$.

\smallskip

We set
\begin{equation*}
\Sigma=\left\{ \xi=(\xi_j)_{j=-\infty}^\infty \in \prod_{j=-\infty}^{\infty} \{1,2,3\} \mid  \ \xi_j\neq\xi_{j+1} \text{ for any $j\in\Z$}\right\}
\end{equation*}
We define the metric $d_\rho:\Sigma\times\Sigma\rightarrow \R$ by 
\begin{equation} \label{e:defdist}
d_\rho(\xi,\eta)=\rho^n \quad \text{if } \xi_j=\eta_j \text{ for } |j|<n \text{ and } \xi_n\neq \eta_n \text{ or } \xi_{-n}\neq \eta_{-n}.
\end{equation}

We say that a point $x\in \Omega_{-}$ solves the itinerary problem
\begin{equation} \label{e:itin}
\xi(y)=\xi\in \Sigma
\end{equation}
if the itinerary $\xi(x)$ of $x$ coincides with the sequence $\xi$.

\begin{lemma}[\cite{morita1991symbolic}]
For any $\xi\in \Sigma$, there exists a unique $x\in \Omega_{-}$ which solves the itinerary problem \eqref{e:itin}. In addition, if we denote by $x(\xi)$ the solution of \eqref{e:itin}, there exists a constant $C>0$ such that
\begin{equation} \label{e:ineqtau}
|\tau_+(x(\xi))-\tau_+(x(\eta))|<Cd_\rho(\xi,\eta) \quad \text{for any } \xi,\eta\in \Sigma,
\end{equation} 
where $\rho=r_0/(1+r_0)$ and $d_\rho$ denotes the metric on $\Sigma$ defined by \eqref{e:defdist}. Moreover, if $r_0$ is sufficiently small, we can take $C=1$ in \eqref{e:ineqtau}.
\end{lemma}

\begin{corollary} \label{c:morita}
Let $n\geq 1$, and $\xi_0,\ldots, \xi_n\in \{1,2,3\}^n$ such that $\xi_j\neq \xi_{j+1}$ for $0\leq j\leq n-1$. Let $\gamma, \gamma'$ be two geodesics such that $\gamma(0),\gamma'(0)\in\mathcal{C}^{\xi_0}$ and whose coding sequences both start with $\xi_0,\xi_1,\ldots,\xi_n$. We denote by $t_n$ (resp. $t_n'$) the time at which $\gamma$ (resp. $\gamma'$) hits $\mathcal{C}^{\xi_n}$. Then
\begin{equation} \label{e:difftemps}
|t_n-t_n'|\leq 3Cr_0
\end{equation}
where $C$ is the same constant as in \eqref{e:ineqtau}.
\end{corollary}
\begin{proof}
We show that if $n$ is even, then $|t_n-t_n'|\leq 2Cr_0$. Together with \eqref{e:ineqtau}, it implies \eqref{e:difftemps} also for odd $n$.

\smallskip

Let $n$ be even. We apply \eqref{e:ineqtau} repeatedly, starting from the middle of the sequence $\xi_0,\ldots, \xi_n$, i.e., from $\xi_{n/2}$. We see that 
\begin{align}
&|(t_{k+1}-t_k)-(t_{k+1}'-t_k')|\leq C\rho^{k+1} \quad \text{ for $0\leq k\leq n/2-1$} \nonumber \\
&|(t_{k+1}-t_k)-(t_{k+1}'-t_k')|\leq C\rho^{n-k} \quad \text{ for $n/2\leq k\leq n-1$}. \nonumber
\end{align}
Summing these inequalities over $0\leq k\leq n-1$ and noting that $t_0=t_0'=0$, we get \eqref{e:ineqtau}.
\end{proof}

Let us now turn to the construction of the geodesic which does not meet $\omega(t)$ on the time interval $[0,T]$. For that, we define the following three ``zones'' of $X$:
\begin{equation*}
Z_1=\conv(\mathcal{C}^2,\mathcal{C}^3), \quad Z_2=\conv(\mathcal{C}^1,\mathcal{C}^3), \quad Z_3=\conv(\mathcal{C}^1,\mathcal{C}^2)
\end{equation*}
where $\conv$ denotes the convex hull (see Figure \ref{f:domain}).

\smallskip

The next lemma roughly says that for any splitting of $[0,T]$ into time intervals and any choice of zones (we associate a zone to each interval of the splitting), there exists a geodesic which, for any time $t\in[0,T]$, lies in the zone associated to the time interval to which $t$ belongs. However, for such a result to be true, we need to relax a bit the time splitting, allowing times of switching from one zone to another to be a bit flexible: this is the role of the times $T_j'$ in the next lemma.
\begin{lemma} \label{l:times}
Fix $n\geq 2$ and real numbers $0=T_0<T_1<\ldots<T_{n-1}<T$ such that $T_{j+1}-T_j\geq 10$ for $0\leq j\leq n-2$. Take also $(a_0,\ldots, a_{n-1})\in\{1,2,3\}^n$ with $a_j\neq a_{j+1}$ for $0\leq j\leq n-2$. Then there exists a geodesic $\gamma$ of $X$ and times $0=T_0'<T_1'<\ldots<T_n'=T$ such that $|T_j'-T_j|\leq 3$ for any $0\leq j\leq n-1$, and $\gamma(t)\in Z_{a_j}$ for any $0\leq j\leq n-1$ and any $t\in [T_j',T_{j+1}']$.
\end{lemma}
\begin{proof}
We will describe $\gamma$ by the successive obstacles it hits for $t\geq 0$, i.e., by its coding sequence restricted to positive times (see Fact \ref{coding}).  The key point is to see that thanks to Corollary \ref{c:morita}, all geodesics sharing a long common subsequence at the beginning of their coding sequence accumulate roughly the same discrepancies (with the terminology introduced in the sketch of proof above).

\smallskip

The algorithm which constructs $\gamma$ through its coding sequence then reads as follows:
\begin{enumerate}
\item Start at time $0$ at a point of $\conv(\Delta)$ which also lies on one of the two circles defining the zone $Z_{a_0}$;
\item For any $0\leq j\leq n-2$, choose a number of bounces between the two circles defining the zone $Z_{a_j}$ such that 
\begin{itemize}
\item at the end of the bounces in zone $Z_{a_j}$, it is possible to go to zone $Z_{a_{j+1}}$ (this imposes the last circle on which the geodesic bounces in zone $Z_{a_j}$);
\item any geodesic with coding sequence starting as prescribed by the successive numbers of bounces in the zones $Z_{a_0},\ldots, Z_{a_j}$ verifies $|T_k'-T_k|\leq 3$ for any $0\leq k\leq j$.
\end{itemize}
\item Then we prescribe that $\gamma$ only bounces between the two circles defining the zone $Z_{a_{n-1}}$.
\end{enumerate}
The only point that needs to be justified is Point 2. In fact, thanks to Corollary \ref{c:morita}, all geodesics starting as prescribed by the successive numbers of bounces in the zones $Z_{a_0},\ldots, Z_{a_{j-1}}$ bounce for the last time in $Z_{a_{j-1}}$ at a time which does not vary more than $3Cr_0\ll 1$ among all these geodesics. Therefore, the number of bounces in zone $Z_{a_{j}}$ that we prescribe is the same for all these geodesics.
\end{proof}

We turn now to the proof of Theorem \ref{t:neg}. We define a map $f:[0,T]\rightarrow \mathcal{P}(\{1,2,3\})$ which associates to any time $t\in [0,T]$ the number(s) of the zone(s) that $\omega(t)$ intersects. Here $\mathcal{P}(\cdot)$ denotes the set of all subsets of a given set. 

\smallskip

Recall that we took $r_0,\e,v\ll 1$.  It follows that it takes a time at least $1/(2v)\gg 1$ for the trajectory $\omega(t)$ to visit all zones $Z_1, Z_2, Z_3$. Said precisely, if $0\leq t_1\leq\ldots\leq t_k\leq T$ verify $\bigcup_{k} f(t_k)=\{1,2,3\}$, then $t_k-t_1\geq 1/(2v)$. 

\smallskip

Hence, it is possible to find an integer $n\geq 2$, real numbers $0=T_0<T_1<\ldots<T_{n-1}<T$ and $(a_0,\ldots, a_{n-1})\in\{1,2,3\}^n$ as in Lemma \ref{l:times}, such that $\omega(t)\cap Z_{a_j}=\emptyset$ for $t\in [T_j-3,T_{j+1}+3]$. Then, the geodesic $\gamma$ constructed in Lemma \ref{l:times} remains at a positive distance of $\omega(t)$ for $t\in [0,T]$, which concludes the proof of Theorem \ref{t:neg}.

\appendix

\section{Proof of Fact \ref{coding}} \label{a:fact}
Fact \ref{coding} is well-known (see \cite[Theorem 0]{morita1991symbolic} for a proof using a minimization argument), but we recall here an elementary proof which explicitly shows the Cantor structure underlying this billiard flow (see Remark \ref{r:BirkhoffSmale}).

\smallskip

Let $(\xi_{n})\in \{0,1,2\}^{\mathbb{N}}$ such that $\xi_j\neq \xi_{j+1}$ for any $j\in\mathbb{N}$. We fix a point $A\in X$ which lies in the convex hull of $\Delta$, the equilateral triangle whose vertices are the centers of $\mathcal{C}^1, \mathcal{C}^2, \mathcal{C}^3$. We only consider geodesics starting at $A$.

\smallskip

Our goal is to find an initial angle $\theta$ such that the coding sequence of the geodesic starting at time $0$ at $A$ and making an initial angle $\theta$ with the horizontal axis is $(\xi_{n})$. Intuitively, this will be done step by step, trying to progressively adjust $\theta$ so that the geodesic first hits $\mathcal{C}^{\xi_{0}}$ (which is the case for plenty of geodesics starting at $A$), then restricting this set of geodesics to those then hitting $\mathcal{C}^{\xi_1}$, $\mathcal{C}^{\xi_2},\ldots$ At each step, the set of possible directions is nonempty, closed and contained in the preceding one. The desired geodesic will then be picked in the intersection of this infinite number of nested sets. In the following paragraphs, we make this intuition precise.

\smallskip

In the sequel, the set of velocities $\mathbb{S}^{1}$ will be identified to $[0,2\pi)$, where all angles considered are taken with respect to an axis which we fix once for all arbitrarily. Given $\eta\in[0,2\pi)$, we denote by $\gamma_{\eta}$ the geodesic starting at $A$ with initial angle $\eta$.

\smallskip

For $i\geq 0$, we set $A_{\xi_{0}\xi_{1}...\xi_{i}}\subset [0,2\pi)$ the set of all angles $\eta$ such that the encoding sequence of $\gamma_{\eta}$ starts with $\xi_{0}, \xi_{1}, ..., \xi_{i}$. Our goal is to prove that 

\begin{equation} \label{intersection}
\underset{i\geq 0}{\bigcap} \ A_{\xi_{0}\xi_{1}...\xi_{i}} \neq \emptyset
\end{equation}
since the encoding sequence of any geodesic in this set is $(\xi_{n})_{n\in\mathbb{N}}$.

\smallskip

We see by continuity of the flow that there exists a non-empty set $[\alpha_{0},\beta_{0}]\subset [0,2\pi)$ such that $\eta\in [\alpha_{0},\beta_{0}]$ if and only if the encoding sequence of $\gamma_{\eta}$ starts with $\xi_{0}$. The set $[\alpha_{0},\beta_{0}]$ is what we called $A_{\xi_{0}}$ and we now know that it is non-empty. Moreover, remark that $\gamma_{\alpha_{0}}$ hits $\mathcal{C}^{\xi_{0}}$ tangently on its left and $\gamma_{\beta_{0}}$ hits $\mathcal{C}^{\xi_{0}}$ tangently on its right. 

\smallskip

When $\eta$ now runs over $[\alpha_{0},\beta_{0}]$, because of this last remark, by continuity of the flow, there exists $[\alpha_{1},\beta_{1}]\subset [\alpha_{0},\beta_{0}]$ such that $\eta\in [\alpha_{1},\beta_{1}]$ if and only if $\gamma_{\eta}$ hits successively $\mathcal{C}^{\xi_{0}}$ and $\mathcal{C}^{\xi_{1}}$. The set $[\alpha_{1},\beta_{1}]$ is what we called $A_{\xi_{0}\xi_{1}}$ and we now know that it is non-empty. Moreover, remark again that $\gamma_{\alpha_{2}}$ necessarily hits $\mathcal{C}^{\xi_{1}}$ tangently on its left and $\gamma_{\beta_{2}}$ necessarily hits $\mathcal{C}^{\xi_{1}}$ tangently on its right. 

\smallskip

Iterating this construction, we define successively the closed sets $A_{\xi_{0}\xi_{1}...\xi_{i}}$ for $i\geq 0$ and we remark that 
\begin{equation} \label{cantorintersection}
\forall i \geq 0, \ \  \ \ A_{\xi_{0}\xi_{1}...\xi_{i}}\neq \emptyset , \ \  \  \ A_{\xi_{0}\xi_{1}...\xi_{i}\xi_{i+1}}\subset A_{\xi_{0}\xi_{1}...\xi_{i}}.
\end{equation}

At step $i$, for each $\eta \in A_{\xi_{0}...\xi_{i}}$, we know that the geodesic $\gamma_{\eta}$ hits successively $\mathcal{C}^{\xi_{0}}$, $\mathcal{C}^{\xi_{1}}, \ldots$ until $\mathcal{C}^{\xi_{i}}$. Moreover, we know that the ``extreme" geodesic $\gamma_{\alpha_{i}}$ (resp., $\gamma_{\beta_{i}}$) hits $\mathcal{C}^{\xi_{i}}$ tangently on its left (resp., on its right). For each $\eta\in [\alpha_{i},\beta_{i}]$, we can look at $\gamma_{\eta}$ at the moment just after it hits $\mathcal{C}^{\xi_{i}}$. It defines a point $q^{i}_{\eta}$ on $\mathcal{C}^{\xi_{i}}$ and a velocity $v^{i}_{\eta}$ pointing outwards $\mathcal{C}^{\xi_{i}}$. The key point which makes the argument work is that $q^{i}_{\alpha_{i}}$ is on the left of  $\mathcal{C}^{\xi_{i}}$ and  $v^{i}_{\alpha_{i}}$ points towards left, whereas $q^{i}_{\beta_{i}}$ is on the right of  $\mathcal{C}^{\xi_{i}}$ and  $v^{i}_{\beta_{i}}$ points towards right. The set $\{ (q^{i}_{\eta}, v_{\eta}^{i}), \ \eta\in [\alpha_{i},\beta_{i}]\}$ defines a connected submanifold of dimension $1$ of the phase space with footpoint in $\mathcal{C}^{\xi_{i}}$. Therefore, by continuity of the flow, when $\eta$ runs over $[\alpha_{i},\beta_{i}]$, it is necessary that there exists $[\alpha_{i+1},\beta_{i+1}]\subset [\alpha_{i},\beta_{i}]$ such that for any $\eta\in [\alpha_{i+1},\beta_{i+1}]$, the couple $(q^{i}_{\eta},v^{i}_{\eta})$ defines a geodesic which continues its path by hitting $\mathcal{C}^{\xi_{i+1}}$.

\smallskip

Using \eqref{cantorintersection}, we immediately get \eqref{intersection}, which concludes the proof of Fact 1.

\begin{remark}\label{r:BirkhoffSmale}
The family of sets $A_{\xi_{0}\xi_{1}...\xi_{i}}$ where $(\xi_{i})$ runs over all sequences $\{0,1,2\}^{\mathbb{N}}$ such that $\xi_j\neq \xi_{j+1}$ for any $j\in\mathbb{N}$ defines a Cantor structure similar to that appearing in the construction of the horseshoe map \cite[Section 2.5]{katok1995introduction}.
\end{remark}

\bigskip
\bigskip

\begin{center}
CYRIL LETROUIT

\medskip
\textit{Sorbonne Universit\'e, Universit\'e Paris-Diderot, CNRS, Inria, Laboratoire Jacques-Louis Lions, F-75005 Paris.}

\textit{DMA, \'Ecole normale sup\'erieure, CNRS, PSL Research University, 75005 Paris.}

\medskip

E-mail: \texttt{cyril.letrouit@ens.fr}
\end{center}

\end{document}